\documentclass[12pt]{amsart}
\usepackage{amssymb,amsmath,amsfonts}
\usepackage{enumerate, dsfont}
\theoremstyle{plain}
 \newtheorem{theorem}{Theorem}
 
 \newtheorem{lemma}{Lemma}
 \newtheorem{corollary}{Corollary}

\theoremstyle{definition}
 \newtheorem{example}{Example}[section]
 \newtheorem{definition}{Definition}[section]
\theoremstyle{remark}
 \numberwithin{equation}{section}

 \DeclareMathOperator{\ordi}{ord_{\infty}}

\begin{document}
\title{Prime Rational Functions over a Field}

\author{Eva Goedhart}
\address{Department of Mathematics, Franklin and Marshall College\\
	Philadelphia, PA US\newline
eva.goedhart@fandm.edu}

\author{OMAR KIHEL}
\address{Department of Mathematics and Statistics, Brock University\\
St. Catharines, L2S 3A1, Canada\\
okihel@brocku.ca}

\author{JESSE LARONE}
\address{Department of Mathematics and Statistics, Brock University\\
St. Catharines, L2S 3A1, Canada\\
jlarone@brocku.ca}



\begin{abstract}
The aim of this paper is to provide sufficient conditions for when a polynomial or rational function over a field $K$ is prime using its order of vanishing at infinity and the resultant.
\end{abstract}

\maketitle

\section{Introduction}
In 1922, J.F.~Ritt~\cite{Ritt} defined composite polynomials $f(x)$ over $\mathbb{C}$ as those for which there exist polynomials $g(x)$ and $h(x)$ in $\mathbb{C}[x]$ such that $f=g\circ h$ where the degrees of $g(x)$ and $h(x)$ are each at least 2. Here the prime polynomials are those for which this does not hold. Engstrom~\cite{En41} and Levi~\cite{Levi} extended Ritt's main results to polynomials over fields of characteristic zero. Though many variations of Ritt's second theorem were published, Zannier~\cite{Za93} finished the work of Ritt and Levi by extending Ritt's second theorem to a field of any characteristic. In 2000, Beardon~\cite{Be01} returned to $\mathbb{C}$ to study the role of critical points and critical values in the prime decomposition of polynomials. A few years later, Ayad~\cite{Ayad} used critical points and critical values of polynomials in $\mathbb{C}[x]$ to determine sufficient conditions for when a polynomial is prime among other results. In 2015, Kihel and Larone~\cite{KiLa} extended the definitions of prime and composite polynomials to rational functions in $\mathbb{C}(x)$. In this paper, we continue this work and extend Ayad's results to find conditions under which polynomials and rational functions are prime over a field $K$.

In Section~\ref{sec:finite}, we briefly discuss the decomposition of polynomial functions over a finite field.  In Section~\ref{sec:field}, for a field $K$ we define and consider prime and composite rational functions in $K(x)$.  We then provide sufficient conditions on the order of vanishing at infinity for $f\in K(x)$ to be prime.  In Section~\ref{sec:resultant}, using the resultant and the discriminant, we give sufficient conditions on the critical values of $f\in K(x)$ for $f$ to be prime.


\section{Functions over $\mathbb{F}_q$}\label{sec:finite}
Let $q$ be a power of a prime number, and denote by $\mathbb{F}_q$ the finite field of $q$ elements.  Any function $\varphi :\mathbb{F}_q\to\mathbb{F}_q$ is equivalent to a polynomial function.  Thus, the set of all functions over $\mathbb{F}_q$ can be described exactly as  the ring $$A=\mathbb{F}_q[x]/(x^q-x)$$ under addition and composition of functions. The units of $A$ with respect to function composition are called the {\it permutation polynomials} of $\mathbb{F}_q$. These are the polynomials that are bijections of $\mathbb{F}_q$ into itself. There are exactly $q!$ such polynomials and they can be computed readily, for example, using Lagrange's interpolation formula.  Explicit classes of permutation polynomials are quite rare, with certain exceptions such as linear polynomials, and as such they form an active area of research.

\begin{theorem}\label{thm:units}
Let $\varphi=\alpha\circ\beta$, where $\varphi,\alpha,\beta\in A$.  Then $\varphi$ is a permutation polynomial if and only if both $\alpha$ and $\beta$ are permutation polynomials.
\end{theorem}
\begin{proof}
First, note that if both $\alpha$ and $\beta$ are permutation polynomials, then so is $\varphi$.  Next, assume that $\varphi$ is a permutation polynomial.  Then $$\vert\mathbb{F}_q\vert=\vert \varphi(\mathbb{F}_q)\vert =\vert \alpha\big( \beta(\mathbb{F}_q)\big)\vert\leq\vert \beta(\mathbb{F}_q)\vert\leq\vert\mathbb{F}_q\vert$$ from which we obtain $\vert \beta(\mathbb{F}_q)\vert=\vert\mathbb{F}_q\vert$.  Since $\beta(\mathbb{F}_q)\subseteq\mathbb{F}_q$, it follows that $\beta(\mathbb{F}_q)=\mathbb{F}_q$, so $\beta$ is a permutation polynomial.  It immediately follows that $\alpha=\varphi\circ \beta^{-1}$ is also a permutation polynomial by the first part. 
\end{proof}

We adopt the standard nomenclature of rings to describe the functions that compose to the zero function in $A$. 

\begin{definition}
Let $\varphi\in A$.  We say that $\varphi$ is a {\it zero divisor} if there exists $\psi\in A\setminus\{0\}$ such that $\psi\circ\varphi=0$.
\end{definition}

\begin{theorem}
Every $\varphi\in  A\setminus\{0\}$ is either a permutation polynomial or a zero divisor.
\end{theorem}
\begin{proof}
For  $\varphi\in  A\setminus\{0\}$, define the map $C_{\varphi}:A\to A$ by $\psi\mapsto\psi\circ\varphi$.  If $C_{\varphi}$ is injective, then it is surjective as well, since $A$ is finite.  Thus, there exists $\psi\in A$ such that $(\psi\circ\varphi)=\mathds{1}$ where $\mathds{1}$ denotes the identity function of $A$. Therefore by Theorem~\ref{thm:units}, $\varphi$ is a permutation polynomial. If $C_{\varphi}$ is not injective, there exist distinct $\alpha,\beta\in A$ such that $\alpha\circ\varphi=\beta\circ\varphi$.  Setting $\psi=\alpha-\beta$ yields $\psi\circ\varphi=(\alpha\circ\varphi)-(\beta\circ\varphi)=0$, showing that $\varphi$ is a zero divisor.
\end{proof}

As a consequence of the preceding result, we claim that there is no interesting notion of irreducibility for the elements of the ring $A$.  The permutation polynomials are units, so they are considered trivial cases which leaves us with the zero divisors of $A$.  Letting $\varphi\in A\setminus\{0\}$ be a zero divisor, there exists $\psi\in A\setminus\{0\}$ such that $\psi\circ\varphi=0$.  Then for $\mathds{1}\in A$, the identity function, $(\psi+\mathds{1})\circ\varphi=\varphi$ is a non-trivial decomposition of $\varphi$ since $\psi+\mathds{1}\neq \mathds{1}$.  In this sense, every function over $\mathbb{F}_q$ is either a permutation polynomial or is a composition of polynomial functions, and so is composite.


\section{Decompositions and the order at infinity}\label{sec:field}
As described in the introduction, there are results regarding the decomposition of polynomials and rational functions into indecomposables, which have been referred to as prime polynomials and prime rational functions.  Ayad \cite{Ayad} studied this problem using valencies for polynomials over $\mathbb{C}$, and Kihel and Larone \cite{KiLa} later considered some rational functions over $\mathbb{C}$.  In this section, we further discuss the topic and improve upon a result of Ayad.  

Throughout, let $K$ be a field and $\overline{K}$ an algebraic closure of $K$.  For any rational function $f\in K(x)\setminus\{0\}$, write $f=f_1/f_2$ with $f_1,f_2\in K[x]$ such that $\gcd (f_1,f_2)=1$, and define as is usual $\deg f=\max\{\deg f_1,\deg f_2\}$.  For convenience, we include the following result regarding the degree of a composition.
\begin{theorem}[\cite{KiLa}, Proposition 2.3]\label{thm:degree} 
Let $g,h\in K(x)$.  Then $\deg (g\circ h)=(\deg g)\cdot (\deg h)$.
\end{theorem}
It is straightforward to show that the rational functions of degree 1 form the group of units in $K(x)$ under the operation of composition.  This immediately justifies the following definitions of prime and composite rational functions, as well as provide a simple class of prime rational functions.
\begin{definition}
Let $L/K$ be a field extension, and let $f\in K(x)$.  If $f=g\circ h$ for some $g,h\in L(x)$ with degrees at least 2, then $f$ is called {\it composite} over $L$.  Otherwise, $f$ is said to be {\it prime} over $L$.
\end{definition}

As a direct corollary to Theorem~\ref{thm:degree}, we have the following. A similar result is indicated in \cite[Corollary 2.4]{KiLa} for complex rational functions. 
\begin{theorem}
Let $f\in K(x)$.  If $\deg f$ is a prime number, then $f$ is prime over $\overline{K}$.
\end{theorem}

We recall the definitions used by Ayad~\cite[Section 2, p.146]{Ayad} over $\mathbb{C}$ here and interpret them over $K$. 
\begin{definition}
Let $f\in K[x]$ and let $a\in\overline{K}$.  The smallest integer $i\geq 1$ such that $i$th derivative at $a$, $f^{(i)}(a)\neq 0$ is called the {\it valency} of $f$ at $a$ and is denoted by $v_f(a)$.  If $v_f(a)\geq 2$, $a$ is called a {\it critical point} of $f$.  An element $b\in\overline{K}$ is called a {\it critical value} of $f$ if there exists a critical point $a$ of $f$ such that $f(a)=b$.
\end{definition}
The above definition translates easily to $K(x)$, so we use it without further modifications.   Under the definition of the valency over $\mathbb{C}$, Ayad proved the following result.
\begin{theorem}[\cite{Ayad}, Theorem 3(1)]\label{thm:valency_C}
Let $f\in \mathbb{C}[x]$, and let $d$ be the greatest proper divisor of $\deg f$.  If $v_f(a)=p$ is prime with $p>d$ from some $a\in\mathbb{C}$, then $f$ is prime.
\end{theorem}
This theorem was proved using a relationship between the valencies of $f$ and the degree of its derivative.  For completeness, we include a general form of this same relationship over $K$. This will help to clarify some of statements that follow it.
\begin{lemma}\label{lem3}
Let $f\in K[x]$.  Then $$(\deg f)-1\geq \sum_{a\in\overline{K}}\big(v_f(a)-1\big)$$ where equality holds when $\deg f\neq 0$.
\end{lemma}
\begin{proof}
Note that the only values of $v_f(a)-1$ which contribute to the sum are those for which $v_f(a)\geq 2$.  In particular, $v_f(a)-1$ counts the multiplicity of $a$ as a root of the first derivative $f^{\prime}$, so the sum cannot exceed the degree of $f^{\prime}$ which is itself less than or equal to $(\deg f)-1$.

Let $f\in K[x]$ such that $\deg f\neq 0$, then either $\text{char}(K)=0$ or $\gcd(\deg f,\text{char}(K))=1$.  In either case, $\deg f^{\prime}=(\deg f)-1$, so the sum of the multiplicities of the roots of $f^{\prime}$ is exactly $(\deg f)-1$, and the sum of $v_f(a)-1$ over all $a\in\overline{K}$ must equal $\deg f^{\prime}=(\deg f)-1$.
\end{proof}

We can refine Theorem \ref{thm:valency_C} by finding a replacement for the valencies.  Since Ayad's definition is in fact related to the order of vanishing at a point, we simply look to that concept instead.  That $v_f(a)-1$ counts the multiplicity of $a$ as a root of $f^{\prime}$ makes it apparent that Lemma \ref{lem3} relates to the sum of the orders of zeros and poles of a rational function over all points in $\mathbb{P}^1$ being 0.

Recalling that for $f\in K(x)\setminus\{0\}$ written as a ratio of polynomials $f=f_1/f_2$, we have $\ordi: K(x)\setminus\{0\}\to\mathbb{Z}$ with $\ordi f=\deg f_2-\deg f_1$. Note that $\vert \ordi f\vert\leq \deg f$ and that $\ordi f$ remains unchanged when simplifying the ratio $f_1/f_2$.  Additionally, for a composite function $f\in K(x)$, we can compose with the identity to obtain a property for $\ordi f$ which is similar to that of Theorem \ref{thm:degree}, which states that the map $\deg: \big(K(x)\setminus\{0\},\circ\big)\to\big(\mathbb{Z},\cdot\big)$ preserves operations.  We first illustrate this with an example.
\begin{example}\label{ex:ordi}
Consider the elements $f,g,h\in \mathbb{C}(x)$ given as $$f(x)=\frac{(x^4+1)^3(x^4+x^2+2)}{(x^2+1)^4},\qquad g(x)=\frac{x+1}{x^4},\qquad \text{and } h(x)=\frac{x^2+1}{x^4+1}.$$  One can readily verify that $f=g\circ h$, yet $-\ordi f=8\neq 3\cdot 2=\ordi g\cdot \ordi h$.  Now, consider a unit function $\mu(x)=1/x$ in $\mathbb{C}(x)$ with $\mu^{-1}=\mu$ and define $G=g\circ\mu^{-1}$ and $H=\mu\circ h$ so that $$G(x)=x^4+x^3\qquad\text{and}\qquad H(x)=\frac{x^4+1}{x^2+1}$$  yield a composition $f=G\circ H$. However, for this new pair of functions we note that $\ordi G\cdot\ordi H=(-4)(-2)=8=-\ordi f$.
\end{example}
The following result details the behaviour observed for the composition $f=G\circ H$ in the example above.
\begin{lemma}\label{lem:ordi}
Let $f\in K(x)$ and $L/K$ be an extension field.  If $f=g\circ h$ for some $g,h\in L(x)$ with $\ordi h<0$, then $-\ordi f=\ordi g\cdot \ordi h$.
\end{lemma}
\begin{proof}
Writing $h=h_1/h_2$ and $$g(x)=b\frac{\prod_{i=1}^{m_1}(x-\alpha_i)}{\prod_{j=1}^{m_2}(x-\beta_j)}$$ with all $\alpha_i,\beta_j\in\overline{K}$ yields $$f(x)=g\big(h(x)\big)=\frac{bh_2(x)^{\deg g-m_1}\prod_{i=1}^{m_1}\big( h_1(x)-\alpha_ih_2(x)\big)}{h_2(x)^{\deg g-m_2}\prod_{j=1}^{m_2}\big(h_1(x)-\beta_jh_2(x)\big)}.$$  Since $\ordi h=\deg h_2-\deg h_1<0$ by assumption, we have
\begin{eqnarray*}
\ordi f&=&\big((\deg g-m_2)\deg h_2+m_2\deg h_1\big)-\big((\deg g-m_1)\deg h_2+m_1\deg h_1\big)\\
&=&(m_1-m_2)(\deg h_2-\deg h_1)\\
&=&-\ordi g\cdot \ordi h
\end{eqnarray*}
as claimed.
\end{proof}
It is beneficial then to have decompositions $f=g\circ h$ for which $\ordi h<0$.  To construct one such decomposition when $\ordi h\geq 0$, write $h=h_1/h_2$ so that dividing $h_1$ by $h_2$ in $K[x]$ yields $a\in K$ and $r\in K[x]$ with $h_1=ah_2+r$, and either $r=0$ or $\deg r<\deg h_2$.  Then $h=h_1/h_2=a+r/h_2$, and we define $\mu(x)=1/(x-a)$.  This yields $f=(g\circ \mu^{-1})\circ (\mu\circ h)$ where $H=\mu\circ h$ satisfies $\ordi H=\ordi (h_2/r)=\deg r-\deg h_2<0$. In Example~\ref{ex:ordi}, notice that for $h_1(x)=x^2+1$ and $h_2(x)=x^4+1$, we have the trivial equation of $h_1(x)=0\cdot h_2(x)+ h_1(x)$ so that $a=0$ and $r=h_1$. 

For a polynomial $f\in K[x]$, $-\ordi f$ and $\deg f$ are equal.  Since $\deg$ retains some of its properties over $K[x]$ when extended to $K(x)$, such as the one given in Theorem \ref{thm:degree}, it is natural to expect $\ordi$ to retain some properties as well.  One such interesting property related to the derivative is noted below.
\begin{theorem}\label{th5}
Let $f\in K(x)$.  If $\ordi f$ is non-zero, then $-\ordi f^{\prime}=(-\ordi f)-1$.
\end{theorem}
\begin{proof}
For $f\in K(x)$ with $\ordi f\neq 0$, we may write $-\ordi f=n_1-n_2$ for some $n_1,n_2\in\mathbb{Z}$ and $$f(x)=\frac{ax^{n_1}+r_1(x)}{x^{n_2}+r_2(x)}$$ where $a\neq 0$ and  $r_1,r_2\in K[x]$ with $\deg r_1<n_1$ and $ \deg r_2<n_2$.  The derivative of $f$ can be obtained by simplifying the product rule $$\frac{\big(an_1x^{n_1-1}+r_1^{\prime}(x)\big)\big(x^{n_2}+r_2(x)\big)-\big(ax^{n_1}+r_1(x)\big)\big(n_2x^{n_2-1}+r_2^{\prime}(x)\big)}{\big(x^{n_2}+r_2(x)\big)^2}$$ which we may write concisely as $$\frac{a(-\ordi f) x^{n_1+n_2-1}+R_1(x)}{x^{2n_2}+R_2(x)}$$ for some $R_1, R_2\in K[x]$ where $\deg R_1<n_1+n_2-1$ and $\deg R_2<2n_2$.  Then $-\ordi f^{\prime}=(n_1+n_2-1)-(2n_2)=n_1-n_2-1=(-\ordi f)-1$.
\end{proof}

Finally, the refinement of Ayad's result~\cite{Ayad} in Theorem~\ref{thm:valency_C} is as follows. 
\begin{theorem}\label{th3}
Let $f\in K(x)$, and let $d$ be the greatest proper divisor of $\deg f$.  If $\ordi f\neq 0$ is divisible by a prime $p>d$, then $f$ is prime over $\overline{K}$.
\end{theorem}
\begin{proof}
Suppose that $f$ is composite.  Then there exist $g,h\in\overline{K}(x)$ with $\ordi h<0$ such that $f=g\circ h$.  By Lemma~\ref{lem:ordi}, either $\ordi h$ or $\ordi g$ is divisible by $p>d$, so that $\vert\ordi h\vert\geq p>d$ or $\vert\ordi g\vert\geq p>d$. However, each also satisfies $\vert\ordi h\vert\leq\deg h\leq d$ or $\vert\ordi g\vert\leq\deg g\leq d$, which is a contradiction in either case.
\end{proof}
This result does indeed extend Theorem \ref{thm:valency_C}. To see this, note that it is assumed that the polynomial $f$ admits a prime $p>d$ and some $a$ such that $v_f(a)=p$.  In such a case, $f-f(a)$ has a zero of order $v_f(a)$ at $a$, and defining the unit $\mu(x)=(ax+1)/x$ gives a rational function $\big(f-f(a)\big)\circ\mu$ with $-\ordi\big( \big(f-f(a)\big)\circ\mu\big)=v_f(a)=p$.  Such a rational function is prime by Theorem \ref{th3}, and so $f$ is prime as well since it can be obtained from the composition with a unit.


\section{Decompositions and the resultant}\label{sec:resultant}
In tis section, we first recall the resultant of two polynomials and the discriminant of a polynomial in $K[x]$.  Then, after we present a theorem relating the discriminant to the resultant of composite functions for the polynomial case, we generalize the notion to two rational functions of $K(x)$ and prove a similar result.  

The {\it resultant} of the two polynomials $f,g\in K[x]$ given by $f(x)=a_nx^n+\cdots+a_0$ and $g(x)=b_mx^m+\cdots+b_0$ is defined as $$\text{Res}_x( f,g)=a_n^m b_m^n\prod_{\alpha,\beta}(\alpha-\beta)$$ where $\alpha, \beta\in\overline{K}$ run over all of the roots of $f$ and $g$, respectively.  Three well-known properties of the resultant are as follows.  Given $f,g,h\in K[x]$, we have
\begin{enumerate}
\item $\text{Res}_x(f,g)=(-1)^{\deg f\deg g}\text{Res}_x(g,f)$,
\item $\text{Res}_x(f,gh)=\text{Res}_x(f,g)\text{Res}_x(f,h)$, and 
\item $\text{Res}_x(f,g)=0$ if and only $f$ and $g$ have a common root in $\overline{K}$.
\end{enumerate}
These properties all follow directly from the definition.  The factorization of one of the polynomials as $g(x)=b_m\displaystyle\prod_{\beta}(x-\beta)$ also shows that 
\begin{enumerate}
\setcounter{enumi}{3}
\item $\text{Res}_x(f,g)=a_n^m\displaystyle\prod_{\alpha}g(\alpha)$.
\end{enumerate}
The above property in turn shows that for a constant $c\in\overline{K}$ we have 
\begin{enumerate}
\setcounter{enumi}{4}
\item $\text{Res}_x(f,cg)=c^n\text{Res}_x(f,g)$.
\end{enumerate}
The {\it discriminant} of $f\in K[x]$ of degree $n$ and leading coefficient $a_n$ is given by $$D[f]=\frac{(-1)^{n(n-1)/2}}{a_n}\text{Res}_X(f,f^{\prime}).$$  Letting $t$ be a new variable, as in Ayad~\cite{Ayad}, $b\in\overline{K}$ is a {\it critical value} of $f$ if and only if it is a root of $D[f-t]$.  The multiplicity of a critical value is defined as its multiplicity as a root of $D[f-t]$, and we call a critical value with multiplicity 1 a {\it simple critical value}.

The following theorem, analogous to Ayad~\cite[Lemma 2]{Ayad}, demonstrates how $D[f-t]$ can be factored.
\begin{theorem}\label{th1}
Let $f=g\circ h$ for $f,g,h\in K[x]$ and let $D[f-t]$ be the discriminant of $f-t$.  There exists $a\in K$ such that $$D[f-t]=aD[g-t]^{\deg h}\text{Res}_x\big(f(x)-t,h^{\prime}(x)\big).$$
\end{theorem}
\begin{proof}
Let $n$ and $a_n$ be the degree and leading coefficient of $f$ respectively.  We have 
\begin{align*}
D[f-t]&=\frac{(-1)^{n(n-1)/2}}{a_n}\text{Res}_x\big(f(x)-t,f^{\prime}(x)\big)\\
&=\frac{(-1)^{n(n-1)/2}}{a_n}\text{Res}_x\Big(g\big(h(x)\big)-t,g^{\prime}\big(h(x)\big)\Big)\text{Res}_x\Big(g\big(h(x)\big)-t,h^{\prime}(x)\Big)\\
\end{align*}
Write $g^{\prime}(x)=\displaystyle (\deg g)b\prod_{\beta}(x-\beta)$ where $\beta\in \overline{K}$ runs over all roots of $g^{\prime}$, and $b\in K$ is the leading coefficient of $g$.  Then the properties (5), (1), and (2) of the resultant can be applied here to yield
\begin{align*}
\text{Res}_x\Big(g\big(h(x)\big)-t,g^{\prime}\big(h(x)\big)\Big)&=\text{Res}_x\Big(g\big(h(x)\big)-t,b(\deg g)\prod_{\beta}\big(h(x)-\beta\big)\Big)\\
&=(b(-1)^{\deg g^{\prime}\deg h}\deg g)^{\deg f}\prod_{\beta}\text{Res}_x\Big(h(x)-\beta,g\big(h(x)\big)-t\Big).
\end{align*}
Letting $c$ denote the leading coefficient of $h(x)$, we have
\begin{align*}
\text{Res}_x\Big(g\big(h(x)\big)-t,g^{\prime}\big(h(x)\big)\Big)&=(b(-1)^{\deg g^{\prime}\deg h}\deg g)^{\deg f}\prod_{\beta}c^{\deg f}\big(g(\beta)-t\big)^{\deg h}\\
&=(b(-1)^{\deg g^{\prime}\deg h}c^{\deg g^{\prime}}\deg g)^{\deg f}\prod_{\beta}\big(g(\beta)-t\big)^{\deg h}.
\end{align*}
From the definition of the discriminant of $g-t$ and the resultant properties (1) and (4), we have 
\begin{align*}
D[g-t]^{\deg h}&=\left[\frac{(-1)^{\deg g((\deg g)-1)/2}}{b}\right]^{\deg h}\Big(\text{Res}_x\big(g(x)-t,g^{\prime}(x)\big)\Big)^{\deg h}\\
&=\left[\frac{(-1)^{\deg g((\deg g)-1)/2}}{b}\right]^{\deg h}(b(-1)^{\deg g^{\prime}}\deg g)^{\deg f}\prod_{\beta}\big(g(\beta)-t\big)^{\deg h}\\
&=(-1)^{\deg f(\deg g-1)/2+\deg f\deg g^{\prime}}b^{\deg f-\deg h}(\deg g)^{\deg f}\prod_{\beta}\big(g(\beta)-t\big)^{\deg h},
\end{align*}
so that 
\begin{align*}
\text{Res}_x\Big(g\big(h(x)\big)-t,g^{\prime}\big(h(x)\big)\Big)&=\frac{(b(-1)^{\deg g^{\prime}\deg h}c^{\deg g^{\prime}}\deg g)^{\deg f}D[g-t]^{\deg h}}{(-1)^{\deg f(\deg g-1)/2+\deg f\deg g^{\prime}}b^{\deg f-\deg h}(\deg g)^{\deg f}}\\
&=(-1)^{n\deg g^{\prime}(\deg h-1)-n(\deg g-1)/2}b^{\deg h}c^{n\deg g^{\prime}}D[g-t]^{\deg h}.
\end{align*}
  Now, setting $a=\displaystyle\frac{(-1)^{n(n-1)/2-n(\deg g-1)/2+n\deg g^{\prime}(\deg h-1)}b^{\deg h}c^{n\deg g^{\prime}}}{a_n}$ yields the desired expression $$D[f-t]=aD[g-t]^{\deg h}\text{Res}_x\Big(g\big(h(x)\big)-t,h^{\prime}(x)\Big).$$
\end{proof}
The next corollary follows directly from Theorem \ref{th1}. We include the proof since Ayad~\cite[Corollary 2]{Ayad} does not.
\begin{corollary}\label{cor:disc}
Let $f\in K[x]$ and let $D(t)$ be the discriminant of $f-t$.  If $f(x)$ is composite with a right composition factor of degree $k$, then there exist polynomials $A,B\in K[t]$ such that $D(t)=A(t)^kB(t)$ and $\deg B\leq k-1$.  Moreover, if $\deg f\neq 0$, then $\deg B=k-1$.
\end{corollary}
\begin{proof}
Let $f,g,h\in K[x]$ such that  $f=g\circ h$ with $\deg h=k$.  By applying Theorem \ref{th1} and setting $B(t)=a\text{Res}_x\Big(g\big(h(x)\big)-t,h^{\prime}(x)\Big)$ and $A(t)=D[g-t]$, we have that $D(t)=A(t)^kB(t)$ with $A,B\in K[t]$.

Further, assuming that $\deg f$ is non-zero, then its divisor $k=\deg h\leq \deg f$ is non-zero as well.  This implies that $\deg B=\deg h^{\prime}=(\deg h)-1=k-1$.
\end{proof}

\begin{theorem}\label{th6}
Let $f\in K[x]$, and let $d$ be the greatest proper divisor of $\deg f$.  If $f$ has at least $d$ simple critical values, then $f$ is prime over $K$.
\end{theorem}
\begin{proof}
Suppose that $f$ is composite with a right composition factor of degree $k$.  Then $2\leq k\leq d$, and we write the discriminant of $f-t$ as $D(t)=A(t)^kB(t)$ as in Corollary~\ref{cor:disc}.  Since $f$ has at least $d$ simple critical values, and each simple critical value is necessarily a root of $B(t)$, we have $k\leq d\leq \deg B\leq k-1$, which is a contradiction.  Thus, $f$ is prime.
\end{proof}

Kihel and Larone \cite{KiLa} studied the resultant of the two rational functions $f,g\in\mathbb{C}(x)$, which they defined as the resultant of their numerators.  This idea translates well to elements of $K(x)$, so we take the same definition in $K(x)$.  The following example illustrates that the result is the expected one that coincides with the polynomial case.
\begin{example}
Let $$f(x)=\frac{(x+1)^4}{x^3}\qquad\text{so}\qquad f^{\prime}(x)=\frac{(x+1)^3(x-3)}{x^4}.$$  The list of all critical points of $f$ are $(-1,0),(-1,0),(-1,0),(3,256/27)$.  Note that $$\text{Res}_x\big(f(x)-t, f^{\prime}(x)\big)=t^3(256-27T)$$ has $0$ and $256/27$ as its roots with multiplicities 3 and 1 respectively.
\end{example}
Kihel and Larone's proofs of Lemma 3.2 and Corollary 3.3 in \cite{KiLa} deal with only $\mathbb{C}$ but immediately translate to $\overline{K}$ under the assumptions made here, yielding the following result analogous to Theorem \ref{th1}.
\begin{theorem}
Let $f\in K(x)$ be composite over $\overline{K}$ with $f=g\circ h$.  There exist an integer $\ell\geq 0$ and polynomials $A,B\in\overline{K}[t]$ such that $$\text{Res}_x(f(x)-t,f^{\prime}(x))=t^{\ell}A(t)^{\deg h}B(t)$$ where $\deg B\geq 2(\deg h)-1$.  Moreover, if $\gcd(\deg f,\ordi f)=1$ and $\ordi f>d$ where $d$ is the greatest proper divisor of $\deg f$, then $\ell>0$.
\end{theorem}
The polynomial $\text{Res}_x(f(x)-t,f^{\prime}(x))$ acts as the analog of the discriminant for rational functions $f\in K(x)$ in the sense that $b\in\overline{K}$ is a critical value of $f$ if and only if $b$ is a root of $\text{Res}_x(f(x)-t,f^{\prime}(x))$.  Extending the nomenclature, we call a critical value with multiplicity 1 as a root of $\text{Res}_x(f(x)-t,f^{\prime}(x))$ a {\it simple critical value} of $f$.

\begin{corollary}
Let $f\in K(x)$, and let $d$ be the greatest proper divisor of $\deg f$.  If $f$ has at least $2d$ non-zero simple critical values, then $f$ is prime over $\overline{K}$.
\end{corollary}
The proof is similar to that of Theorem \ref{th6} with instead the inequality $$2\deg h\leq 2d\leq \deg B\leq2\deg h-1$$ yielding the contradiction.

\end{document}